\documentclass[12 pt]{amsart}
\usepackage{amsfonts}
\usepackage{amssymb}
\usepackage{amsmath,amscd}
\usepackage{tikz}
\usepackage{centernot} 
\usepackage{pb-diagram} 
\usepackage{mathrsfs}
\usepackage[OT2,T1]{fontenc}
\usepackage{qtree}

\usepackage{mathtools}
\mathtoolsset{showonlyrefs}

\makeindex 

\def\Z{{\rm \mathbb{Z}}}

\def\Q{{\rm \mathbb{Q}}}

\def\C{{\rm \mathbb{C}}}

\def\P{{\rm \mathbb{P}}}
\def\H{{\rm \mathbb{H}}}
\def\p{{\rm \mathfrak{p}}}

\def\a{{\rm \mathfrak{a}}}
\def\b{{\rm \mathfrak{b}}}

\def\O{{\rm \mathcal{O}}}

\def\lcm{{\rm lcm}}
\def\cond{{\rm cond}}
\def\Aut{{\rm Aut}}
\def\Ext{{\rm Ext}}

\def\Pic{{\rm Pic}}

\def\GL{{\rm GL}}

\def\an{{\rm an}}

\def\Hom{{\rm Hom}}
\def\End{{\rm End}}

\def\kk{f_{\ki}}
\def\dd{f_{\di}}
\def\ki{1}
\def\di{2}

\numberwithin{equation}{section}

\newtheorem{theorem}{Theorem}[section]
\newtheorem{lemma}[theorem]{Lemma}

\newtheorem{corollary}[theorem]{Corollary}
\newtheorem{proposition}[theorem]{Proposition}
\newenvironment{definition}[1][Definition]{\begin{trivlist}
\item[\hskip \labelsep {\bfseries #1}]}{\end{trivlist}}

\newenvironment{remark}[1][Remark]{\begin{trivlist}
\item[\hskip \labelsep {\bfseries #1}]}{\end{trivlist}}

\def\pu{\phi_u}

 \def\ltwomat{\left[\begin{array}{cc}}       \def\rtwomat{\end{array}\right]}

 \def\hom{\mathrm{Hom}}
\def\tL{\widetilde{L}}

\def\L{\Lambda}

\begin{document}
\title[Extensions of CM curves and orbit counting]{Extensions of CM elliptic curves and\\ orbit counting on the projective line}
\author{Julian Rosen and Ariel Shnidman}
\maketitle
%
%
%

\begin{abstract}
There are several formulas for the number of orbits of the projective line under the action of subgroups of $\GL_2$. We give an interpretation of two such formulas in terms of the geometry of elliptic curves, and prove a more general formula for a large class of congruence subgroups of Bianchi groups.  Our formula involves the number of walks on a certain graph called an isogeny volcano. Underlying our results is a complete description of the group of extensions of a pair of CM elliptic curves, and of a pair of lattices in a quadratic field.    
\end{abstract}

\section{Introduction}
We begin with some motivation from two well-known and elegant formulas. The first formula is for the number of orbits of $\P^1(\Q)$ under the action of the congruence subgroup $\Gamma_0(N) \subset \GL_2(\Z)$,\footnote{Our definition of $\Gamma_0(N)$ is not the traditional one, as we allow elements of determinant -1.} acting by fractional-linear transformation:
\begin{equation}\label{modular}
\#\Gamma_0(N) \backslash \P^1(\Q) = \sum_{d | N} \pu\left(\gcd(d,N/d)\right).
\end{equation}
Here $\phi_u$ is the \emph{reduced totient function}
\[
\pu(n):=\#\left(\frac{(\Z/n\Z)^\times}{\{\pm1\}}\right)=\begin{cases}\frac{\phi(n)}{2}&\text{ if }n\geq 3,\\1&\text{ if }n=1,2.\end{cases}
\]

The second formula is for the number of orbits of $\P^1(K)$ under the action of $\GL_2(\O_{K})$, where $K$ is a number field and $\O_{K}$ is its ring of integers:
\begin{equation}\label{bianchi}
\#\GL_2(\O_{K}) \backslash \P^1(K) = h.
\end{equation}  
Here, $h = h(K)$ is the size of the class group $\Pic(\O_{K})$.  Formula \eqref{bianchi} is due to Bianchi. A formula counting the number of orbits of $\P^1(K)$ under the action of congruence subgroups analogous to $\Gamma_0(N)$ is given in \cite{CA}.

We can connect these two formulas using the theory of elliptic curves.  For this, we let $E$ and $E'$ be elliptic curves over $\C$, and we consider the number $N(E,E')$ of elliptic curves on the abelian surface $A = E \times E'$ up to the action of $\Aut(A)$.  This number is finite by \cite{lenzar}, and in fact $N(E,E') = 2$ unless there exists an isogeny $\lambda : E \to E'$.  This raises the question, how do we compute $N(E,E')$ if $E$ and $E'$ are isogenous?

If $\End(E) = \Z$, then $\Hom(E,E') = \lambda \Z$, for a certain minimal isogeny $\lambda$, whose degree we will denote by $N$. Then $\Aut(A) \simeq \Gamma_0(N)$ and we have $N(E,E') =\#\Gamma_0(N) \backslash \P^1(\Q)$ \cite[Prop.\ 3.7]{paper1}.  So the number $N(E,E')$ is given by the first formula \eqref{modular}.   On the other hand, if $E$ has complex multiplication (CM) by an imaginary quadratic field $K$, we may think of $\Aut(A)$ as a subgroup of $\GL_2(K)$.  As before, we have $N(E,E') =\# \Aut(A) \backslash \P^1(K)$; see Lemma \ref{bij}.   In the special case where $E = E'$ and $\End(E) = \O_{K}$, we have $\Aut(A) \simeq \GL_2(\O_{K})$ and Bianchi's formula \eqref{bianchi} gives $N(E,E) = h$.

Our main result is a formula for $N(E,E')$ for any two elliptic curves $E,E'$ with CM by $K$.  Equivalently, we compute $\#\Aut(M) \backslash \P^1(K)$ for any lattice $M \subset K^2$.  To state the result, we define the {\it conductor} of an elliptic curve $E$, with CM by $K$, to be the index of $\End(E)$ inside the ring of integers $\O_{K}$.  Thus, if $E$ has conductor $c$, then $\End(E)$ is isomorphic to $\O_c$, the unique subring of index $c$ inside $\O_K$.  Concretely, if $E \simeq \C/\a$, for some lattice $\a$ in $K$, then $c$ is the index of the ring of multipliers $\{ \alpha \in K : \alpha \a \subset \a\}$ in $\O_K$.  

\begin{theorem}\label{main}
Let $K$ be an imaginary quadratic field whose only roots of unity are $\pm 1$.\footnote{There are analogous formulas for the two imaginary quadratic fields with more roots of unity, but we omit these cases for simplicity.}
Suppose $E$ and $E'$ are elliptic curves with complex multiplication by $K$ and of conductors $c$ and $c'$ respectively.  Define $f=\lcm(c,c')$ and $f'=\gcd(c,c')$.
Then 
\[N(E,E') = \sum_{\ell | f} h_\ell \cdot \pu\left(f/\ell\right) \sum_{g | \ell} 2^{\omega(\ell/g)}r_{K}\left(f', g, f/\ell\right).\]
Here:
\begin{itemize}
\item $\displaystyle h_\ell = \#\Pic(\O_\ell) = h\cdot \ell \prod_{p | \ell} \left(1 - \chi_{K}(p)/p\right),$\\ 
\item $\chi_{K}$ is the quadratic Dirichlet character associated to $K$, 
%
 \item $\omega(n)$ is the number of distinct prime factors of $n$, and
 
  \item  $r_{K}(a,b,N)$ is the number of cyclic subgroups of order $N$ of a fixed elliptic curve of conductor $a$ such that the quotient has conductor $b$.
  
 \end{itemize}
\end{theorem}

The mysterious quantity in our formula is the function $r_K(a,b,N)$.  This is a 3-variable multiplicative function which is made completely explicit in Corollary \ref{explicit}.  For each prime $p$, the values of $r_K$ on powers of $p$ depend only on how $p$ splits in $K$, so the number $N(E,E')$ depends only on the two integers $c$ and $c'$ and the values of $\chi_K$ on primes dividing $cc'$.  The explicit formulas for $r_K(a,b,N)$ are complicated, but we may interpret these numbers as the number of walks on certain graphs called {\it $p$-isogeny volcanoes}.  The structure of the $p$-isogeny volcano makes it easy to compute the numbers $r_K(a,b,N)$, just by looking at the graph (see Theorem \ref{thformula}). A script for computing $N(E,E')$ in Sage is included with this document on arXiv.

As a corollary of Theorem \ref{main}, we obtain orbit counting formulas for large class of subgroups of $\GL_2(K)$ which are commensurable with the Bianchi group $\GL_2(\O_K)$.  Explicitly, if $E = \C/\a$ and $E' = \C/\a'$, for lattices $\a$ and $\a'$ in $K$, then $\Aut(E \times E')$ is isomorphic to the group:
\[\Gamma(\a, \a') := \left\{ \gamma  \in \left(\begin{array}{cc}\O_c & c\a(f'\a')^{-1}\\ c'\a'(f'\a)^{-1}&\O_{c'} \end{array}\right) : \det \gamma  = \pm 1\right\}. \]
We then have the following orbit counting formula:
\begin{corollary}\label{maingroup} 
\[\#\Gamma(\a,\a')\backslash \P^1(K) = \sum_{\ell | f} h_\ell \cdot \pu\left(f/\ell\right) \sum_{g | \ell} 2^{\omega(\ell/g)}r_{K}\left(f', g, f/\ell\right).\]
\end{corollary}

In favorable cases, the formula in Corollary \ref{maingroup} simplifies.  For example, if $\a = \a' = \O_f$, then $\Gamma(\a,\a') = \GL_2(\O_f)$ and the right hand side becomes a simple Dirichlet convolution:

\begin{corollary}\label{gl2}
For any $f \geq 1$, 
\[\#\GL_2(\O_f)\backslash \P^1(K) = \sum_{\ell | f} h_\ell \cdot \pu(f/\ell).\]
\end{corollary}

\begin{proof}
See Section \ref{volcano}. 
\end{proof}

These results have application to other counting problems in geometry.  For example, $\#\Gamma(\a,\a')\backslash \P^1(K)$ is the number of cusps on the hyperbolic 3-manifold $\Gamma(\a,\a')\backslash \H^3$.  It is also the number of equivalence classes of contractions of the abelian surface $A = E \times E'$, in the sense of the minimal model program.  Our formula can be used to study the asymptotics of these quantities as $A$ varies, and would be helpful in individual computations as well.    

The proof of Theorem \ref{main} involves a careful study of $\Ext$-groups in the category of products of elliptic curves with complex multiplication by $K$, or equivalently, in the category of lattices in imaginary quadratic fields. These results (found in Section\ \ref{extensions}) are interesting in their own right and should find other applications.

\subsection{Acknowledgements} The second author thanks Andrew Snowden for a helpful conversation.  The second author was partially supported by NSF grant DMS-0943832.   

\section{Extensions of CM elliptic curves}\label{extensions}
Fix an imaginary quadratic number field $K\subset\C$. A \emph{$K$-lattice} is a free abelian subgroup $L\subset K^n$ of rank $2n$ (for some $n$). The quotient $\C^n/L$ is a complex torus, which is known to be algebraic.

\subsection{Singular abelian surfaces}
For context, we recall a basic fact about the abelian surfaces we are considering.  

\begin{proposition}[\cite{SM}]
\label{propequiv}
Let $A/\C$ be an abelian surface.  Then the following are equivalent:
\begin{enumerate}
\item $A \simeq \C^2/\Lambda$, with $\Lambda$ a $K$-lattice, for some imaginary quadratic field $K$.

\item $A$ is the product of two elliptic curves having CM by the same field $K$. 

\item $A$ is isogenous to a product of two elliptic curves having CM by the same field $K$.
\end{enumerate}
\end{proposition}
\noindent If these equivalent conditions hold, we say that $A$ is {\it singular} and that $A$ has \emph{CM by $K$}.   



\subsection{Rank $1$ $K$-lattices}
Recall that an \emph{order} in the quadratic field $K$ is a subring $R\subset\O_{K}$ such that $\mathrm{Frac}(R)=K$. Every order has the form $R=\Z+f\O_{K}$ for a unique positive integer $f$, called the \emph{conductor} of $R$. For each lattice $L \subset K^n$, the \emph{ring of multipliers} $R(L)$ is the order $\{ \alpha \in K : \alpha L \subset L\}$.  The conductor of $L$ is defined to be the conductor of $R(L)$.

If $\a \subset K$ is a rank 1 $K$-lattice, then $\a$ is projective as an $R(\a)$-module. Two rank 1 $K$-lattices $\a$ and $\a'$ are \emph{homothetic} if $\a = \gamma \a'$ for some $\gamma \in K^\times$.  The set of homothety classes of lattices of conductor $f$ forms a group under multiplication of lattices, which is denoted $\Pic(\O_f)$.  The set of homothety classes of lattices in $K$ is therefore in bijection with $\coprod_{f \geq 1} \Pic(\O_f)$.  If $E = \C/\a$, then we define the conductor of $E$, denoted $\cond(E)$, to be the conductor of $\a$. 

\begin{proposition}
\label{propcan}
Let $A$ be a singular abelian surface with CM by $K$. Then there is a positive integer $f$ and a lattice $\a\subset K$ of conductor $f'$ dividing $f$ such that
\[
A \cong \C/\O_f \oplus \C/\a.
\]
Moreover, the integers $f$, $f'$ and the class $[\a] \in \Pic(\O_{f'})$ are uniquely determined by these conditions. In particular, 
\[
A \mapsto (f/f', [\a])
\]
is a bijection between the set of isomorphism classes of singular abelian surfaces with CM by $K$ and pairs $(g, [\a])$, where $g \geq 1$ and $[\a]$ is a homothety class of lattices in $K$.  
\end{proposition}

\begin{proof}
Since $A$ is singular with CM by $K$, we may write $A = E \times E'$ with $E = \C/\a$ and $E' = \C/\a'$ for rank 1 $K$-lattices $\a$ and $\a'$.  Let $c$ and $c'$ be the conductors of these elliptic curves.  Then 
$E \times E' \simeq \C/\O_f \times \C/\a\a'$, where $f = \lcm(c_1,c_2)$.  Moreover, if two lattices $\a, \b \subset K$ have multiplication by $\O_f$, then $\C^2/(\O_f \oplus \a) \simeq \C^2/(\O_f \oplus \b)$ if and only if $\a$ and $\b$ are homothetic (see \cite{SM} or \cite{kaniCM}).      
\end{proof}


\subsection{Extensions of rank $1$ lattices}
\label{secext}
Let $L_{\ki}, L_{\di} \subset K$ be lattices of conductors $f_{\ki}$ and $f_{\di}$, and let $E_i = \C/L_i$ ( for $i=1,2$) be the corresponding elliptic curves.  Suppose also that $L_{\ki}=\Z+\Z\tau_{\ki}$ and $L_{\di}=\Z+\Z\tau_{\di}$ for elements $\tau_{\ki}$ and $\tau_{\di}$ of $K$; up to homothety, we may always choose such a basis.
\begin{proposition}\cite[I.5.7 and I.6.2]{bl}
\label{propmat}
The association
\begin{equation}
z\mapsto\frac{\C^2}{\begin{pmatrix} \tau_{\di} & 1 & z & 0\\ 0 & 0 & \tau_{\ki} & 1\end{pmatrix}}
\end{equation}
induces an isomorphism between the group $\tilde E := \C/L_{\ki}L_{\di}$ and the group $\Ext^1_\an(E_{\ki},E_{\di})$ of equivalence classes of extensions of complex tori.  Under this bijection, the subgroup of torsion points in $\tilde E$ corresponds to the subgroup $\Ext^1_{alg}(E_{\ki},E_{\di})$ of algebraic extensions.
\end{proposition}  

\begin{remark}
There is a similar result in \cite[Thm.\ 6.1]{papram}, attributed to Lichtenbaum, but the result is not stated correctly there.   
\end{remark}

Proposition \ref{secext} shows that $\Ext^1_{alg}(E_{\ki},E_{\di}) \simeq (\Q/\Z)^2$ as a group.  But it is not clear which (or how many) extension classes correspond to some fixed abelian surface. The following theorem gives this extra information.
\begin{theorem}\label{order}
Let $P \in \tilde E$ be a torsion point of order $n \geq 1$, and let
\[0 \to E_{\di} \to S \to E_{\ki} \to 0\]
be the corresponding extension given by Proposition $\ref{secext}$.  Then \[
S \simeq \C/\O_{nf} \oplus \tilde E/\langle P\rangle,
\] where $f= \lcm(\kk,\dd)$.    
\end{theorem}
In the proof, we consider  several different notions of extension.
\begin{itemize}
\item For each positive integer $F$ divisible by $f$, the lattices $L_{\ki}$ and $L_{\di}$ can be considered as $\O_F$-modules, and we have the group $\Ext^1_{\O_F}(L_{\ki},L_{\di})$ of extensions of $\O_F$-modules
\item We have the group $\Ext^1(L_{\ki},L_{\di})$ of extensions of $K$-lattices
\item We have the group $\Ext^1_{alg}(\C/L_{\ki},\C/L_{\di})$ of extensions  of abelian varieties
\item We have the group $\Ext^1_{an}(\C/L_{\ki},\C/L_{\di})$ of extensions of complex tori.
\end{itemize}
An extension of modules determines an extension of $K$-lattices, which determines an extension of abelian varieties, which determines an extension of complex tori, so there is a sequence of group homomorphisms
\begin{gather}
\label{gammamap}\Ext^1_{\O_F}(L_{\ki},L_{\di})\xrightarrow{\gamma_{1}} \Ext^1(L_{\ki},L_{\di})\xrightarrow{\gamma_{2}}\Ext^1_{alg}(\C/L_{\ki},\C/L_{\di})\\\xrightarrow{\gamma_3}\Ext^1_{an}(\C/L_{\ki},\C/L_{\di}).
\end{gather}
It is not hard to see that $\gamma_{1}$, $\gamma_{2}$, and $\gamma_3$ are injective. It is well-known that $\gamma_{2}$ is an isomorphism; indeed, there is an equivalence of categories between $K$-lattices and abelian varieties isogenous to a product of elliptic curves with CM by $K$. Proposition \ref{propmat} implies the image of $\gamma_3$ is the torsion subgroup of $\Ext^1_{an}(\C/L_{\ki},\C/L_{\di})$.

\begin{lemma}\label{lem}
For any $F$ divisible by $f=\lcm(\kk,\dd)$, there is an isomorphism of $\O_F$-modules
\[
\Ext^1_{\O_F}(L_{\ki},L_{\di})\cong\frac{\hom_{\O_F}(L_{\ki},L_{\di})}{\frac{F}{f} \hom_{\O_F}(L_{\ki},L_{\di})}.
\]
\end{lemma}

\begin{proof}
Choose an algebraic integer $\omega$ with $\O_{K}=\Z+\omega\Z$; then $\O_F=\Z+F\omega\Z$. 
Tensoring with $\O_{\kk}$ induces a surjection $\Pic(\O_F)\to\Pic(\O_{\kk})$, so we can find a sublattice $\tL_{\ki}\subset L_{\ki}$ with ring of multipliers $\O_F$ such that  $\O_{\kk}\cdot  \tL_{\ki}=L_{\ki}$. Consider the following resolution of $L_{\ki}$ by projective $\O_F$-modules:

\[0\longleftarrow L_{\ki}\xleftarrow{\hspace{2mm}\varphi_0}\tL_{\ki}^2\xleftarrow{\hspace{2mm}\varphi_{\ki}}\tL_{\ki}^2\xleftarrow{\hspace{2mm}\varphi_{\di}}\ldots,\]
with 
\[\varphi_0=[1,-k\omega]  \hspace{5mm} \mbox{and }\hspace{3mm }\varphi_{\ki}=\varphi_{\di}=\ldots=\ltwomat F\omega&-\kk F\omega^2\\\frac{F}{\kk}&-F\omega\rtwomat .\]
We apply the functor $\hom(-,L_{\di})$ and examine the first coordinate to obtain
\begin{align*}
\Ext^1_{\O_F}(L_{\ki},L_{\di})&=\frac{\bigg\{\alpha\in\hom_{\O_F}(\tL_{\ki},L_{\di}):\O_{\kk}\alpha\subset\hom_f(\tL_{\ki},L_{\di}) \bigg\}}{\bigg\{\frac{F}{\kk}\O_{\kk}\cdot\hom_{\O_F}(\tL_{\ki},L_{\di})\bigg\}}\\
&=\frac{\hom_{\O_F}(L_{\ki},L_{\di})}{\frac{F}{f}\hom_{\O_F}(L_{\ki},L_{\di})},
\end{align*}
where the second line follows from \cite[Lem.\ 15]{kaniCM}.
\end{proof}

\begin{remark}
This lemma holds if $K$ is a {\it real} quadratic field as well. 
\end{remark}

\begin{corollary}
\label{lemar}
The map $\gamma_{i}$ of \eqref{gammamap} takes $\Ext^1_{\O_F}(L_{\ki},L_{\di})$ isomorphically onto the $F/f$-torsion in $\Ext^1(L_{\ki},L_{\di})$.
\end{corollary}
\begin{proof}
Lemma \ref{lem} shows that $\Ext^1_{\O_F}(L_{\ki},L_{\di})$ is an $F/f$-torsion group, so $\gamma_1$ maps into the $F/f$-torsion in $\Ext^1(L_{\ki},L_{\di})$. Proposition \ref{propmat} implies that the $F/f$-torsion in $\Ext^1_{alg}(\C/L_{\ki}, \C/L_2)$ has cardinality $(F/f)^2$, hence the $F/f$-torsion in $\Ext^1(L_{\ki},L_{\di})$ also has cardinality $(F/f)^2$ because $\gamma_2$ is an isomorphism. Finally, we know $\gamma_{1}$ is injective, and since its source and target have the same finite cardinality, $\gamma_1$ must be an isomorphism.
\end{proof}

\begin{lemma}
\label{lemorder}
Suppose
\[
0\to L_{\di}\to L\to L_{\ki}\to 0
\]
is an extension of $K$-lattices, where $L_i$ has conductor $f_i$ (for $i=1,2$). If the corresponding element of $\Ext^1(L_{\ki},L_{\di})$ has order $n$, then $L$ has conductor $n\cdot \lcm(\kk,\dd)$.
\end{lemma}
\begin{proof}
The conductor of $L$ is the minimal $F$ such that the class in $\Ext^1(L_{\ki},L_{\di})$ representing $L$ is in the image of $\Ext^1_{\O_F}(L_{\ki},L_{\di})$. Corollary \ref{lemar} implies that this value is $n\cdot \lcm(\kk,\dd)$.
\end{proof}

\begin{proof}[Proof of Theorem $\ref{order}$]
The abelian surface $S$ is given as $\C^2/L$, where $L$ is the $\Z$-span of the period matrix in Proposition \ref{secext}, with $z \in \C$ any lift of the order $n$ torsion point $P \in \tilde E = \C/L_{\ki}L_{\di}$.  On the other hand, by Proposition \ref{propcan}, there is an integer $N\geq 1$ and a lattice $\a$ with $\O_N\subset R(\a)$, such that $L \simeq \O_N \oplus \a$.  The conductor of $\O_N\oplus \a$ is $N$, so Lemma \ref{lemorder} implies $N=n f$.

We can recover $\a$ as the quotient of the group 
\[\bigwedge_{\O_N}\nolimits^{2} \bigg(\O_N \oplus \a \bigg) \cong \bigwedge_{\O_N}\nolimits^{2} \a \,  \oplus \, \bigwedge_{\O_N}\nolimits^{2}\O_N \, \oplus \big(\O_N \otimes_{\O_N}\a\big) \cong \left(\bigwedge_{\O_N}\nolimits^{2} \a\right) \oplus \a\]
modulo it torsion.
On the other hand, the torsion-free part of $\bigwedge^2_{\O_n}L$ is spanned by the $2\times2$-minors of the period matrix.  Thus, $\a$ is the lattice generated by $L_{\ki}L_{\di}$ and the element $z$, i.e., the lattice corresponding to the elliptic curve $\tilde E/\langle P\rangle$.  
\end{proof}

\section{Proof of Theorem \ref{main}}
Let $A = E \times E'$ be a product of two elliptic curves with CM by the same imaginary quadratic field $K$, and suppose the conductor of $A$ is $f\geq 1$.  Then we may think of $\Aut(A)$ as a subgroup of $\GL_2(K)$.  The latter acts on $\P^1(K)$ by fractional linear transformation.      

\begin{lemma}\label{bij}
The orbits on $\P^1(K)$ under the action of $\Aut(A)\subset GL_2(K)$ are in bijection with the $\Aut(A)$-orbits of elliptic curves contained in $A$.  
\end{lemma}

\begin{proof}
An elliptic curve $E\subset A$ determines a $1$-dimensional subspace $T_0 E\subset T_0 A$, and the map $E\mapsto T_0 E$ is an $\Aut(A)$-equivariant bijection between elliptic curves on $A$ and $1$-dimensional subspaces of $T_0A$.
\end{proof}

Thus, to prove Theorem \ref{main}, it suffices to count the number $N(E,E')$ of $\Aut(A)$-equivalence classes of elliptic curves on the abelian surface $A = E \times E'$.  By Proposition \ref{propcan}, we may assume $E = \C/\O_f$ and $E' = \C/\a$, where $\a$ has conductor $f'$ dividing $f$.  If $F \subset A$ is an elliptic curve contained in $A$, then both $F$ and the quotient $A/F$ are elliptic curves with CM by $K$, and one has a short exact sequence of abelian varieties:
\begin{equation}\label{ext}0 \to F \to A \to A/F \to 0.\end{equation}   
If $F' \subset A$ is another elliptic curve with corresponding sequence
\begin{equation}\label{ext'}0 \to f' \to A \to A/f' \to 0,\end{equation}   
then $F\subset A $ and $F' \subset A$ are $\Aut(A)$-equivalent if and only if \eqref{ext} and \eqref{ext'} are isomorphic as short exact sequences.

\begin{lemma}
If $F \subset A$ is an elliptic curve, then the conductor of $F$ divides $f$. 
\end{lemma}
\begin{proof}
Choose a map $j: E \to A$ so that $F = j(E)$.  Write $j = (j_{\ki}, j_{\di})$ for maps $j_{\ki} : E \to E$ and $j_{\di}: E \to \C/\a$.  Since $\O_f$ acts on both $E$ and $\C/\a$, we have $\ker j_{\ki} = E[I_{\ki}]$ and $\ker j_{\di} = E[I_{\di}]$ for certain ideals $I_{\ki}$ and $I_{\di}$ of $\O_f$, by a result of Kani \cite[Thm.\ 20(b)]{kaniCM}. Here, $E[I]$ is the subgroup of points in the kernel of $\alpha : E \to E$ for every $\alpha \in I$.  Thus, 
\[\ker j = \ker j_{\ki} \cap \ker j_{\di} = E[I_{\ki}] \cap E[I_{\di}] = E[I_{\ki} + I_{\di}].\]  It follows that $F \simeq E/\ker j \simeq E/E[I_{\ki} + I_{\di}]$, and so $F$ has multiplication by $\O_f$. 
\end{proof}

\begin{proposition}\label{extred}
$N(E,E')$ is equal to the number of isomorphism classes of short exact sequences \[0 \to E_{\ki} \to A \to E_{\di} \to 0,\]
 with $E_{\ki}$ and $E_{\di}$ elliptic curves of conductor dividing $f$.      
\end{proposition}

\begin{proof}
The previous lemma shows that any elliptic curve $E_{\ki} \subset A$ has conductor dividing $f$.  Dualizing, we see that $E_{\di} = A/E_{\ki}$ is also an elliptic curve on $\hat A = A$, so has conductor dividing $f$ as well.  
\end{proof}

We continue the proof of Theorem \ref{main}. Suppose that $E_{\ki}$ and $E_{\di}$ are elliptic curves with conductor dividing $f$.  Observe that if $\gamma_{\ki},\gamma_{\di}\in\Ext^1(E_{\di},E_{\ki})$ correspond to extensions
\begin{equation}
\label{e1}
0\to E_{\ki}\to A_{\ki}\to E_{\di}\to 0,
\end{equation}
\begin{equation}
\label{e2}
0\to E_{\ki}\to A_{\di}\to E_{\di}\to 0,
\end{equation}
then \eqref{e1} and \eqref{e2} are isomorphic as short exact sequences if and only if $\gamma_{\ki}$ and $\gamma_{\di}$ are in the same orbit of $\Aut(E_{\ki})\times \Aut(E_{\di})$ on $\Ext^1(E_{\di},E_{\ki})$. By assumption, $K$ contains no non-trivial roots of unity, so $\Aut(E_{\ki})=\Aut(E_{\di})=\{\pm 1\}$.
Combining this observation with Proposition \ref{extred}, we obtain
\[
N(E,E') =  \#\Aut(A)\backslash \P^1(K) = \sum_{\kk, \dd | f}\sum_{\substack{L_{\ki} \in \Pic(\O_{\kk})\\L_{\di} \in \Pic(\O_{\dd})}}  \# \Ext^1_{\O_f}(L_{\di},L_{\ki})_A/\{\pm 1\},
\]
where $\Ext^1_{\O_f}(L_{\di},L_{\ki})_A$ is the set of extensions classes $0 \to L_{\ki} \to L \to L_{\di} \to 0$ in $\Ext^1_{\O_f}(L_{\di},L_{\ki})$ such that $A \simeq \C^2/L$.  

Now fix integers $\kk$ and $\dd$ dividing $f$, and set $g=\gcd(\kk,\dd)$, $\ell=\lcm(\kk,\dd)$.  Recall that an isogeny of elliptic curves is \emph{cyclic} if its kernel is a cyclic group. We call a cyclic isogeny \emph{based} if the kernel is equipped with a distinguished generator. By Theorem \ref{order}, classes in $\Ext^1_{\O_f}(L_{\di},L_{\ki})_A$ correspond to based cyclic isogenies $\C/L_{\ki}L_{\di}\to \C/\a$ of degree $f/\ell$. Dualizing, we find that these are equinumerous to based cyclic $(f/\ell)$-isogenies $\C/\a \to \C/L_{\ki}L_{\di}$.

Recall that $r_K(a,b,N)$ is the number of cyclic subgroups of order $N$ of a fixed elliptic curve of conductor $a$ such that the quotient has conductor $b$. Thus, there are $r_K(f',g,f/\ell)$ cyclic subgroups $G\subset \C/\a$ of order $f/\ell$ such that the quotient has conductor $g$. For each such subgroup, there are $h_{\kk} h_{\dd}/h_g=h_\ell$ pairs $L_{\ki}\in\Pic(\O_{\kk})$, $L_{\di}\in\Pic(\O_{\dd})$ with $(\C/\a)/G\cong \C/L_{\ki}L_{\di}$. Given such $G$, $L_{\ki}$, and $L_{\di}$, there are $\pu(f/\ell)$ based cyclic $(f/\ell)$-isogenies $\C/\a\to \C/L_{\ki}L_{\di}$ with kernel $G$, up to the action of $\{\pm1\}$.  So we conclude that     
\begin{align*}
\#\Aut(A)\backslash \P^1(K) &= \sum_{\substack{k,d\,  | f\\ g = \gcd(k,d)\\ \ell = \lcm(k,d)}} h_\ell\cdot \pu\left(f/\ell\right) \cdot r_K\left(f',g,f/\ell\right)\\
&= \sum_{g \, | \, \ell\,  | f} 2^{\omega(\ell/g)} \cdot h_\ell\cdot \pu\left(f/\ell\right) \cdot r_K\left(f',g,f/\ell\right),\\
\end{align*}
completing the proof of Theorem \ref{main}.

\section{Computation of $r_K(a,b,c)$}\label{volcano}
\label{secr}
To make Theorem \ref{main} explicit, we need to compute the number $r_K(a,b,c)$ of cyclic subgroups $C$ of order $c$ of a fixed CM elliptic curve $E$ of conductor $a$ such that $E/C$ has conductor $b$.  We will see that this does not depend on the choice of the elliptic curve $E$ of conductor $a$.  First we reduce to the case where $a,b$, and $c$ are powers of a prime $p$:

\begin{lemma}\label{mult}
The function $r_K(a,b,c)$ is multiplicative: if we factor $a$, $b$, and $c$ into prime powers $a = \prod_p p^{a_p}$, $b = \prod_p p^{b_p}$, and $c = \prod p^{c_p}$, then 
\[r_K(a,b,c) = \prod_p r_K\left(p^{a_p}, p^{b_p}, p^{c_p}\right).\]
\end{lemma}  

\begin{proof}
If $E$ is an elliptic curve, then giving a cyclic subgroup of order $c$ is the same as giving a cyclic subgroup of order $p^{c_p}$ for every prime $p$.  Furthermore, if $E \to E'$ is an isogeny of CM elliptic curves of $p$-power degree, the ratio of the conductors of $E$ and $E'$ is a power of $p$.  To see this, it suffices to show that if $\a \subset \b$ is an inclusion of rank 1 $K$-lattices of index $p^a$, then the conductors of $\a$ and $\b$ are off by a power of $p$.  And indeed, the $\ell$-part of the conductor can be computed locally at $\ell$, and $\a \hookrightarrow \b$ induces an isomorphism $\a \otimes \Z_\ell \simeq  \b \otimes \Z_\ell$, for all $\ell \neq p$.  It follows then that 
\[r_K(a,b,c) = \prod_p r_K(a,p^{b_p}\prod_{\ell \neq p} \ell^{a_\ell}, p^{c_p})\]
and 
\[ r_K(a, p^{b_p} \prod_{\ell \neq p} \ell^{a_\ell}, p^{c_p}) =  r_K(p^{a_p},p^{b_p},p^{c_p}),\]
which proves the lemma.     
\end{proof}

Next, we fix a prime $p$ and show how to compute the numbers $r_K(p^a,p^b,p^c)$.  We will use a variant of the $p$-isogeny volcano \cite{suther}, a tool typically used to study isogenies of elliptic curves over finite fields.  We define a graph $G_p$, called the {\it volcano}, whose structure depends only on $p$ and $\chi_K(p)$, i.e.\ whether $p$ is split, inert, or ramified in $K$.  We omit the dependance on $K$ in the notation.  The vertices of $G_p$ are partitioned into infinitely many {\it levels} $G_{p,k}$, one for each integer $k \geq 0$.  We call the subgraph on $G_{p,0}$ the {\it rim} of the volcano.  The graph $G_p$ is then uniquely determined by the following conditions:
\begin{itemize}
\item $G_p$ is $(p+1)$-regular, with no self-loops or multi-edges.  
\item The rim $G_{p,0}$ is the complete graph on $2 + \chi_K(p)$ vertices.
\item For $k \geq 1$, each vertex in $G_{p,k}$ has a unique neighbor in $G_{p, k-1}$, and this accounts for every edge not in the rim.      
\end{itemize}

\noindent
The figures below depict the top two levels of the graph $G_3$ for each of the three splitting types. 
\begin{figure}[htp]
\begin{tikzpicture}

\draw (-1.4,-0.29) -- (-3.30,-1.5);
\draw[fill=red] (-3.30,-1.5) circle (0.06);
\draw (-1.4,-0.29) -- (-2.0,-1.5);
\draw[fill=red] (-2.0,-1.5) circle (0.06);
\draw (-1.4,-0.29) -- (-.70,-1.5);
\draw[fill=red] (-.70,-1.5) circle (0.06);
\draw (-1.4,-0.29) -- (.60,-1.5);
\draw[fill=red] (.60,-1.5) circle (0.06);
\draw[fill=red] (-1.4,-0.29) circle (0.06);

\draw (-3.30,-1.5) -- (-3.9,-3.0);
\draw[fill=red] (-3.90,-3.0) circle (0.06);
\draw (-3.90,-3.0) -- (-4.0,-3.7);
\draw (-3.90,-3.0) -- (-3.9,-3.7);
\draw (-3.90,-3.0) -- (-3.8,-3.7);

\draw (-3.30,-1.5) -- (-3.45,-3.0);
\draw[fill=red] (-3.45,-3.0) circle (0.06);
\draw (-3.45,-3.0) -- (-3.55,-3.7);
\draw (-3.45,-3.0) -- (-3.45,-3.7);
\draw (-3.45,-3.0) -- (-3.35,-3.7);

\draw (-3.30,-1.5) -- (-2.9,-3.0);
\draw[fill=red] (-2.9,-3.0) circle (0.06);
\draw (-2.9,-3.0) -- (-3.0,-3.7);
\draw (-2.9,-3.0) -- (-2.9,-3.7);
\draw (-2.9,-3.0) -- (-2.8,-3.7);

\draw (-2.0,-1.5) -- (-2.6,-3.0);
\draw[fill=red] (-2.60,-3.0) circle (0.06);
\draw (-2.6,-3.0) -- (-2.7,-3.7);
\draw (-2.6,-3.0) -- (-2.6,-3.7);
\draw (-2.6,-3.0) -- (-2.5,-3.7);

\draw (-2.0,-1.5) -- (-2.1,-3.0);
\draw[fill=red] (-2.1,-3.0) circle (0.06);
\draw (-2.1,-3.0) -- (-2.2,-3.7);
\draw (-2.1,-3.0) -- (-2.1,-3.7);
\draw (-2.1,-3.0) -- (-2.0,-3.7);

\draw (-2.0,-1.5) -- (-1.65,-3.0);
\draw[fill=red] (-1.65,-3.0) circle (0.06);
\draw (-1.65,-3.0) -- (-1.75,-3.7);
\draw (-1.65,-3.0) -- (-1.65,-3.7);
\draw (-1.65,-3.0) -- (-1.55,-3.7);

\draw (-.7,-1.5) -- (-1.2,-3.0);
\draw[fill=red] (-1.20,-3.0) circle (0.06);
\draw (-1.2,-3.0) -- (-1.3,-3.7);
\draw (-1.2,-3.0) -- (-1.2,-3.7);
\draw (-1.2,-3.0) -- (-1.1,-3.7);

\draw (-.7,-1.5) -- (-.7,-3.0);
\draw[fill=red] (-.7,-3.0) circle (0.06);
\draw (-.7,-3.0) -- (-.8,-3.7);
\draw (-.7,-3.0) -- (-.7,-3.7);
\draw (-.7,-3.0) -- (-.6,-3.7);

\draw (-.7,-1.5) -- (-.2,-3.0);
\draw[fill=red] (-.2,-3.0) circle (0.06);
\draw (-.2,-3.0) -- (-.3,-3.7);
\draw (-.2,-3.0) -- (-.2,-3.7);
\draw (-.2,-3.0) -- (-.1,-3.7);

\draw (.6,-1.5) -- (.2,-3.0);
\draw[fill=red] (.2,-3.0) circle (0.06);
\draw (.2,-3.0) -- (.1,-3.7);
\draw (.2,-3.0) -- (.2,-3.7);
\draw (.2,-3.0) -- (.3,-3.7);

\draw (.6,-1.5) -- (.65,-3.0);
\draw[fill=red] (.65,-3.0) circle (0.06);
\draw (.65,-3.0) -- (.55,-3.7);
\draw (.65,-3.0) -- (.65,-3.7);
\draw (.65,-3.0) -- (.75,-3.7);

\draw (.6,-1.5) -- (1.3,-3.0);
\draw[fill=red] (1.3,-3.0) circle (0.06);
\draw (1.3,-3.0) -- (1.2,-3.7);
\draw (1.3,-3.0) -- (1.3,-3.7);
\draw (1.3,-3.0) -- (1.4,-3.7);

\end{tikzpicture}
\caption{$G_3$, inert case.}\label{figure:inert}
\end{figure}
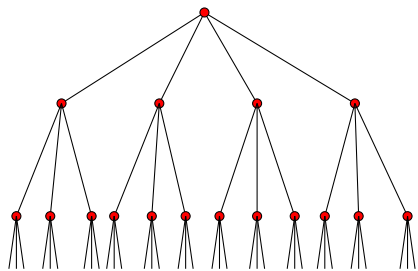
\vspace{-8pt}

\hspace{4mm} \newline

\begin{figure}[htp]
\begin{tikzpicture}

\draw (-2.0,-0.29) -- (-3.30,-1.5);
\draw[fill=red] (-3.30,-1.5) circle (0.06);
\draw (-2.0,-0.29) -- (-2.0,-1.5);
\draw[fill=red] (-2.0,-1.5) circle (0.06);
\draw (-2.0,-0.29) -- (-.70,-1.5);
\draw[fill=red] (-.70,-1.5) circle (0.06);
\draw[fill=red] (-2.0,-0.29) circle (0.06);

\draw (-3.30,-1.5) -- (-3.9,-3.0);
\draw[fill=red] (-3.90,-3.0) circle (0.06);
\draw (-3.90,-3.0) -- (-4.0,-3.7);
\draw (-3.90,-3.0) -- (-3.9,-3.7);
\draw (-3.90,-3.0) -- (-3.8,-3.7);

\draw (-3.30,-1.5) -- (-3.45,-3.0);
\draw[fill=red] (-3.45,-3.0) circle (0.06);
\draw (-3.45,-3.0) -- (-3.55,-3.7);
\draw (-3.45,-3.0) -- (-3.45,-3.7);
\draw (-3.45,-3.0) -- (-3.35,-3.7);

\draw (-3.30,-1.5) -- (-2.9,-3.0);
\draw[fill=red] (-2.9,-3.0) circle (0.06);
\draw (-2.9,-3.0) -- (-3.0,-3.7);
\draw (-2.9,-3.0) -- (-2.9,-3.7);
\draw (-2.9,-3.0) -- (-2.8,-3.7);

\draw (-2.0,-1.5) -- (-2.6,-3.0);
\draw[fill=red] (-2.60,-3.0) circle (0.06);
\draw (-2.6,-3.0) -- (-2.7,-3.7);
\draw (-2.6,-3.0) -- (-2.6,-3.7);
\draw (-2.6,-3.0) -- (-2.5,-3.7);

\draw (-2.0,-1.5) -- (-2.1,-3.0);
\draw[fill=red] (-2.1,-3.0) circle (0.06);
\draw (-2.1,-3.0) -- (-2.2,-3.7);
\draw (-2.1,-3.0) -- (-2.1,-3.7);
\draw (-2.1,-3.0) -- (-2.0,-3.7);

\draw (-2.0,-1.5) -- (-1.65,-3.0);
\draw[fill=red] (-1.65,-3.0) circle (0.06);
\draw (-1.65,-3.0) -- (-1.75,-3.7);
\draw (-1.65,-3.0) -- (-1.65,-3.7);
\draw (-1.65,-3.0) -- (-1.55,-3.7);

\draw (-.7,-1.5) -- (-1.2,-3.0);
\draw[fill=red] (-1.20,-3.0) circle (0.06);
\draw (-1.2,-3.0) -- (-1.3,-3.7);
\draw (-1.2,-3.0) -- (-1.2,-3.7);
\draw (-1.2,-3.0) -- (-1.1,-3.7);

\draw (-.7,-1.5) -- (-.7,-3.0);
\draw[fill=red] (-.7,-3.0) circle (0.06);
\draw (-.7,-3.0) -- (-.8,-3.7);
\draw (-.7,-3.0) -- (-.7,-3.7);
\draw (-.7,-3.0) -- (-.6,-3.7);

\draw (-.7,-1.5) -- (-.2,-3.0);
\draw[fill=red] (-.2,-3.0) circle (0.06);
\draw (-.2,-3.0) -- (-.3,-3.7);
\draw (-.2,-3.0) -- (-.2,-3.7);
\draw (-.2,-3.0) -- (-.1,-3.7);

\def \x {4.2}

\draw (\x-2.0,-0.29) -- (\x-3.30,-1.5);
\draw[fill=red] (\x-3.30,-1.5) circle (0.06);
\draw (\x-2.0,-0.29) -- (\x-2.0,-1.5);
\draw[fill=red] (\x-2.0,-1.5) circle (0.06);
\draw (\x-2.0,-0.29) -- (\x-.70,-1.5);
\draw[fill=red] (\x-.70,-1.5) circle (0.06);
\draw[fill=red] (\x-2.0,-0.29) circle (0.06);
\draw (\x-2.0,-0.29) -- (-2.0,-0.29);

\draw (\x-3.30,-1.5) -- (\x-3.9,-3.0);
\draw[fill=red] (\x-3.90,-3.0) circle (0.06);
\draw (\x-3.90,-3.0) -- (\x-4.0,-3.7);
\draw (\x-3.90,-3.0) -- (\x-3.9,-3.7);
\draw (\x-3.90,-3.0) -- (\x-3.8,-3.7);

\draw (\x-3.30,-1.5) -- (\x-3.45,-3.0);
\draw[fill=red] (\x-3.45,-3.0) circle (0.06);
\draw (\x-3.45,-3.0) -- (\x-3.55,-3.7);
\draw (\x-3.45,-3.0) -- (\x-3.45,-3.7);
\draw (\x-3.45,-3.0) -- (\x-3.35,-3.7);

\draw (\x-3.30,-1.5) -- (\x-2.9,-3.0);
\draw[fill=red] (\x-2.9,-3.0) circle (0.06);
\draw (\x-2.9,-3.0) -- (\x-3.0,-3.7);
\draw (\x-2.9,-3.0) -- (\x-2.9,-3.7);
\draw (\x-2.9,-3.0) -- (\x-2.8,-3.7);

\draw (\x-2.0,-1.5) -- (\x-2.6,-3.0);
\draw[fill=red] (\x-2.60,-3.0) circle (0.06);
\draw (\x-2.6,-3.0) -- (\x-2.7,-3.7);
\draw (\x-2.6,-3.0) -- (\x-2.6,-3.7);
\draw (\x-2.6,-3.0) -- (\x-2.5,-3.7);

\draw (\x-2.0,-1.5) -- (\x-2.1,-3.0);
\draw[fill=red] (\x-2.1,-3.0) circle (0.06);
\draw (\x-2.1,-3.0) -- (\x-2.2,-3.7);
\draw (\x-2.1,-3.0) -- (\x-2.1,-3.7);
\draw (\x-2.1,-3.0) -- (\x-2.0,-3.7);

\draw (\x-2.0,-1.5) -- (\x-1.65,-3.0);
\draw[fill=red] (\x-1.65,-3.0) circle (0.06);
\draw (\x-1.65,-3.0) -- (\x-1.75,-3.7);
\draw (\x-1.65,-3.0) -- (\x-1.65,-3.7);
\draw (\x-1.65,-3.0) -- (\x-1.55,-3.7);

\draw (\x-.7,-1.5) -- (\x-1.2,-3.0);
\draw[fill=red] (\x-1.20,-3.0) circle (0.06);
\draw (\x-1.2,-3.0) -- (\x-1.3,-3.7);
\draw (\x-1.2,-3.0) -- (\x-1.2,-3.7);
\draw (\x-1.2,-3.0) -- (\x-1.1,-3.7);

\draw (\x-.7,-1.5) -- (\x-.7,-3.0);
\draw[fill=red] (\x-.7,-3.0) circle (0.06);
\draw (\x-.7,-3.0) -- (\x-.8,-3.7);
\draw (\x-.7,-3.0) -- (\x-.7,-3.7);
\draw (\x-.7,-3.0) -- (\x-.6,-3.7);

\draw (\x-.7,-1.5) -- (\x-.2,-3.0);
\draw[fill=red] (\x-.2,-3.0) circle (0.06);
\draw (\x-.2,-3.0) -- (\x-.3,-3.7);
\draw (\x-.2,-3.0) -- (\x-.2,-3.7);
\draw (\x-.2,-3.0) -- (\x-.1,-3.7);

%
%

\end{tikzpicture}
\caption{$G_3$, ramified case.}\label{ramified}
\end{figure}
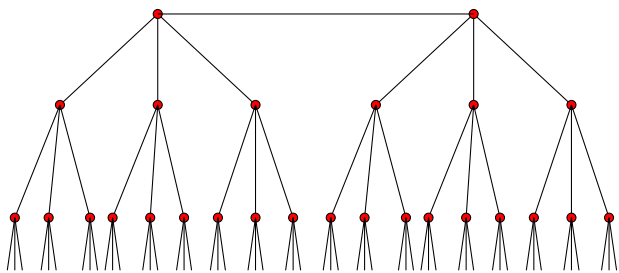
\vspace{-8pt}

\hspace{4mm} \newline

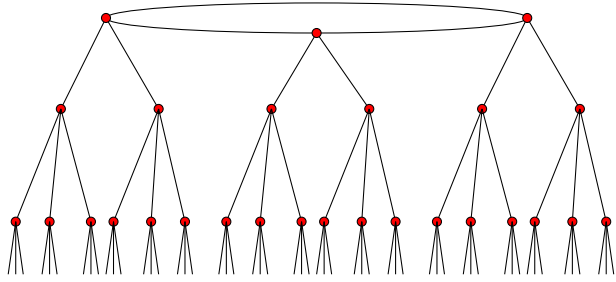
\begin{figure}[htp]
\begin{tikzpicture}

\draw (0.1,-.29) ellipse (2.8 and 0.2);

\draw (-2.7,-0.29) -- (-3.30,-1.5);
\draw[fill=red] (-3.30,-1.5) circle (0.06);
\draw (-2.7,-0.29) -- (-2.0,-1.5);
\draw[fill=red] (-2.0,-1.5) circle (0.06);
\draw[fill=red] (-2.7,-0.29) circle (0.06);

\draw (-3.30,-1.5) -- (-3.9,-3.0);
\draw[fill=red] (-3.90,-3.0) circle (0.06);
\draw (-3.90,-3.0) -- (-4.0,-3.7);
\draw (-3.90,-3.0) -- (-3.9,-3.7);
\draw (-3.90,-3.0) -- (-3.8,-3.7);

\draw (-3.30,-1.5) -- (-3.45,-3.0);
\draw[fill=red] (-3.45,-3.0) circle (0.06);
\draw (-3.45,-3.0) -- (-3.55,-3.7);
\draw (-3.45,-3.0) -- (-3.45,-3.7);
\draw (-3.45,-3.0) -- (-3.35,-3.7);

\draw (-3.30,-1.5) -- (-2.9,-3.0);
\draw[fill=red] (-2.9,-3.0) circle (0.06);
\draw (-2.9,-3.0) -- (-3.0,-3.7);
\draw (-2.9,-3.0) -- (-2.9,-3.7);
\draw (-2.9,-3.0) -- (-2.8,-3.7);

\draw (-2.0,-1.5) -- (-2.6,-3.0);
\draw[fill=red] (-2.60,-3.0) circle (0.06);
\draw (-2.6,-3.0) -- (-2.7,-3.7);
\draw (-2.6,-3.0) -- (-2.6,-3.7);
\draw (-2.6,-3.0) -- (-2.5,-3.7);

\draw (-2.0,-1.5) -- (-2.1,-3.0);
\draw[fill=red] (-2.1,-3.0) circle (0.06);
\draw (-2.1,-3.0) -- (-2.2,-3.7);
\draw (-2.1,-3.0) -- (-2.1,-3.7);
\draw (-2.1,-3.0) -- (-2.0,-3.7);

\draw (-2.0,-1.5) -- (-1.65,-3.0);
\draw[fill=red] (-1.65,-3.0) circle (0.06);
\draw (-1.65,-3.0) -- (-1.75,-3.7);
\draw (-1.65,-3.0) -- (-1.65,-3.7);
\draw (-1.65,-3.0) -- (-1.55,-3.7);

%
%

\def \x {2.8}

\draw (\x-2.7,-0.49) -- (\x-3.30,-1.5);
\draw[fill=red] (\x-3.30,-1.5) circle (0.06);
\draw (\x-2.7,-0.49) -- (\x-2.0,-1.5);
\draw[fill=red] (\x-2.0,-1.5) circle (0.06);
\draw[fill=red] (\x-2.7,-0.49) circle (0.06);

\draw (\x-3.30,-1.5) -- (\x-3.9,-3.0);
\draw[fill=red] (\x-3.90,-3.0) circle (0.06);
\draw (\x-3.90,-3.0) -- (\x-4.0,-3.7);
\draw (\x-3.90,-3.0) -- (\x-3.9,-3.7);
\draw (\x-3.90,-3.0) -- (\x-3.8,-3.7);

\draw (\x-3.30,-1.5) -- (\x-3.45,-3.0);
\draw[fill=red] (\x-3.45,-3.0) circle (0.06);
\draw (\x-3.45,-3.0) -- (\x-3.55,-3.7);
\draw (\x-3.45,-3.0) -- (\x-3.45,-3.7);
\draw (\x-3.45,-3.0) -- (\x-3.35,-3.7);

\draw (\x-3.30,-1.5) -- (\x-2.9,-3.0);
\draw[fill=red] (\x-2.9,-3.0) circle (0.06);
\draw (\x-2.9,-3.0) -- (\x-3.0,-3.7);
\draw (\x-2.9,-3.0) -- (\x-2.9,-3.7);
\draw (\x-2.9,-3.0) -- (\x-2.8,-3.7);

\draw (\x-2.0,-1.5) -- (\x-2.6,-3.0);
\draw[fill=red] (\x-2.60,-3.0) circle (0.06);
\draw (\x-2.6,-3.0) -- (\x-2.7,-3.7);
\draw (\x-2.6,-3.0) -- (\x-2.6,-3.7);
\draw (\x-2.6,-3.0) -- (\x-2.5,-3.7);

\draw (\x-2.0,-1.5) -- (\x-2.1,-3.0);
\draw[fill=red] (\x-2.1,-3.0) circle (0.06);
\draw (\x-2.1,-3.0) -- (\x-2.2,-3.7);
\draw (\x-2.1,-3.0) -- (\x-2.1,-3.7);
\draw (\x-2.1,-3.0) -- (\x-2.0,-3.7);

\draw (\x-2.0,-1.5) -- (\x-1.65,-3.0);
\draw[fill=red] (\x-1.65,-3.0) circle (0.06);
\draw (\x-1.65,-3.0) -- (\x-1.75,-3.7);
\draw (\x-1.65,-3.0) -- (\x-1.65,-3.7);
\draw (\x-1.65,-3.0) -- (\x-1.55,-3.7);

\def \xx {5.6}

\draw (\xx-2.7,-0.29) -- (\xx-3.30,-1.5);
\draw[fill=red] (\xx-3.30,-1.5) circle (0.06);
\draw (\xx-2.7,-0.29) -- (\xx-2.0,-1.5);
\draw[fill=red] (\xx-2.0,-1.5) circle (0.06);
\draw[fill=red] (\xx-2.7,-0.29) circle (0.06);

\draw (\xx-3.30,-1.5) -- (\xx-3.9,-3.0);
\draw[fill=red] (\xx-3.90,-3.0) circle (0.06);
\draw (\xx-3.90,-3.0) -- (\xx-4.0,-3.7);
\draw (\xx-3.90,-3.0) -- (\xx-3.9,-3.7);
\draw (\xx-3.90,-3.0) -- (\xx-3.8,-3.7);

\draw (\xx-3.30,-1.5) -- (\xx-3.45,-3.0);
\draw[fill=red] (\xx-3.45,-3.0) circle (0.06);
\draw (\xx-3.45,-3.0) -- (\xx-3.55,-3.7);
\draw (\xx-3.45,-3.0) -- (\xx-3.45,-3.7);
\draw (\xx-3.45,-3.0) -- (\xx-3.35,-3.7);

\draw (\xx-3.30,-1.5) -- (\xx-2.9,-3.0);
\draw[fill=red] (\xx-2.9,-3.0) circle (0.06);
\draw (\xx-2.9,-3.0) -- (\xx-3.0,-3.7);
\draw (\xx-2.9,-3.0) -- (\xx-2.9,-3.7);
\draw (\xx-2.9,-3.0) -- (\xx-2.8,-3.7);

\draw (\xx-2.0,-1.5) -- (\xx-2.6,-3.0);
\draw[fill=red] (\xx-2.60,-3.0) circle (0.06);
\draw (\xx-2.6,-3.0) -- (\xx-2.7,-3.7);
\draw (\xx-2.6,-3.0) -- (\xx-2.6,-3.7);
\draw (\xx-2.6,-3.0) -- (\xx-2.5,-3.7);

\draw (\xx-2.0,-1.5) -- (\xx-2.1,-3.0);
\draw[fill=red] (\xx-2.1,-3.0) circle (0.06);
\draw (\xx-2.1,-3.0) -- (\xx-2.2,-3.7);
\draw (\xx-2.1,-3.0) -- (\xx-2.1,-3.7);
\draw (\xx-2.1,-3.0) -- (\xx-2.0,-3.7);

\draw (\xx-2.0,-1.5) -- (\xx-1.65,-3.0);
\draw[fill=red] (\xx-1.65,-3.0) circle (0.06);
\draw (\xx-1.65,-3.0) -- (\xx-1.75,-3.7);
\draw (\xx-1.65,-3.0) -- (\xx-1.65,-3.7);
\draw (\xx-1.65,-3.0) -- (\xx-1.55,-3.7);

%
%

\end{tikzpicture}
\caption{$G_3$, split case.}\label{ramified}
\end{figure}
\vspace{-8pt}

The vertices in $G_p$ of level $k$ are meant to correspond to elliptic curves with conductor equal to $p^k$.  This is not a bijective correspondence however; the graph $G_p$ is merely a convenient tool to understand the network of $p$-isogenies between CM elliptic curves.            

Recall that if $G = (V,E)$ is a graph, and $v,v' \in V$ are vertices, then a {\it walk from $v$ to $v'$ of length $c$} is a sequence  
\[v = v_0, e_1, v_1, e_{\di}, \cdots , v_{c-1}, e_c, v_c = v',\] 
where each $e_i \in E$ is an edge connecting $v_{i-1}$ and $v_i$.      

\begin{definition}
A walk is called {\it non-backtracking} if $e_{i} \neq e_{i+1}$ for all $i$.
\end{definition}

\begin{theorem}
\label{thformula}
Let $a,b,c$ be non-negative integers.  Then $r_K(p^a, p^b,p^c)$ is the number of non-backtracking walks of length $c$ in $G_p$ starting at a {\it fixed} vertex of level $a$ and ending at a vertex of level $b$.    
\end{theorem}

\begin{proof}
First consider the case $c = 0$. Then $r_K(p^a, p^b, 1) = 0$ unless $a = b$, in which case it equals 1.  This agrees with the fact that there are no walks of length 0 from a fixed vertex of level $a$ to level $b$, unless $a = b$, in which case there is exactly one: the trivial walk.  

Next consider the case $c = 1$.  If $E$ is an elliptic curve of conductor $p^a$, then $E$ has $p+1$ subgroups $C$ of order $p$, and Theorem \ref{thformula} is claiming that the conductors $\cond(E/C)$ are exactly given by the levels of the $p+1$ neighbors of any vertex of level $a$ in $G_p$.  Specifically, if $a = 0$, the claim is that exactly $1 + \chi_K(p)$ of these quotients have conductor $1$, while the others have conductor $p$. And if $a > 0$, the claim is that exactly one of these quotients has conductor $p^{a-1}$, while the others have conductor $p^{a+1}$.  

We give a short proof of these two facts using Kani's result \cite[Thm.\ 20(b)]{kaniCM} that a finite subgroup $C \subset E$ is of the form $C = E[I]$ for some ideal $I \subset \End(E)$ if and only if the conductor of $E/C$ divides the conductor of $E$.  Thus, if $E$ has conductor $1$, then exactly $1 + \chi_K(p)$ subgroups of order $p$ give rise to quotients of conductor $1$, namely the subgroups of the form $E[\p]$ for an ideal $\p \subset \O_K$ of norm $p$.  For the other subgroups $C$, the ring $\End(E/C)$ certainly contains $\O_p = \Z + p\O_K$, and so this containment must be a strict equality: $\End(E/C) = \O_p$.  

If $E$ has conductor $p^a$ with $a > 0$, then the non-Dedekind ring $\End(E) \simeq \O_{p^a}$ has exactly one ideal of norm $p$: the ideal $I= p\O_{p^{a-1}}$.  If $E \simeq \C/L$, then $E/E[I] \simeq \C/\O_{p^{a-1}}L$, which has conductor $p^{a-1}$.  Thus, by Kani's result, $E$ has exactly one $p$-isogeny to a curve of conductor $p^{a-1}$, and no $p$-isogenies to curves of conductor $p^a$.  The remaining $p$ subgroups $C$ have the property that $\O_{p^{a+1}} = \Z + p\O_{p^a} \subset \End(E/C)$, but by the above, this inclusion must be an equality, i.e.\ $E/C$ has conductor $p^{a+1}$.  This proves the theorem in the case $c = 1$.            

The general case $c > 1$ follows immediately from the case $c = 1$, since every cyclic $p^c$-isogeny is a composition of $p$-isogenies.  Composing $p$-isogenies amounts to taking walks along the isogeny volcano $G_p$.  The condition that the walk is non-backtracking comes from the condition that the $p^c$-isogeny is cyclic.  Indeed, backtracking amounts to composing with the dual of the previous $p$-isogeny, which would make the the composition divisible by $p$ and hence not cyclic.      
\end{proof}


Theorem \ref{thformula} allows us to compute $r_K(p^a,p^b,p^c)$ explicitly:
\begin{corollary}\label{explicit}
Let $\chi_K(p)$ equal $-1$, $0$ , or $1$, depending on whether $p$ is inert, ramified, or split in $K$.  Then 
\[r_K(p^a,p^b,p^c) = 
\begin{cases}
 0 & \mbox{if } a < b  \mbox{ and } c < b-a \\
p^c & \mbox{if } 0 < a \leq b \mbox{ and } c = b-a\\
\left(p - \chi_K(p)\right)p^{c -1} & \mbox{if } 0 = a < b \mbox{ and } c = b\\
(p-1)p^{(b-a + c)/2 - 1} & \mbox{if } a \leq b \mbox{ and } b-a < c < b + a\\ 
(p - \chi_K(p) - 1)p^{b-1} & \mbox{if } a \leq b \mbox{ and }c = b + a\\
\left(1 + \chi_K(p)\right)\left(p - \chi_K(p)\right)p^{b-1} & \mbox{if } a \leq b \mbox{ and }c = b + a + 1\\
\left(\chi_K(p) + |\chi_K(p)|\right)(p-1)p^{b-1} & \mbox{if } a \leq b \mbox{ and }c > b + a + 1\\
r_K(p^b,p^b,p^{c-a+b}) & \mbox{if } a > b.\\
\end{cases}\]
Here, we use the convention that $p^s = 0$ if $s$ is not an integer.  Note that the formula in the last case reduces to one of the earlier cases.  
\end{corollary}      
\begin{remark}
We will see in the proof that counting non-backtracking walks on $G_p$ is quite easy; it is only the formulas which are complicated.   
\end{remark}

\begin{proof}
The structure of $G_p$ makes it simple to compute the number of non-backtracking walks from level $a$ to level $b$ of length $c$.  Indeed, once a non-backtracking walk descends down the volcano, then it can never reascend.  So any non-backtracking walk starting at level $a$ consists of three stages.  It begins by ascending up the volcano to some level $k \leq a$, and there is a unique way to do this.  

If $k = 0$, then there is a middle stage where the walk may move horizontally along the rim.  In the split case, this may involve circling around the rim many times, but the walk is not allowed to change directions (by the non-backtracking assumption).  There are therefore two ways to perform this middle stage in the split case, no matter how many times the walk winds around the rim.  In the ramified case, the walk may traverse the unique edge on the rim once, but then it is forced to move down the volcano, again by the non-backtracking assumption.  In the inert case, this middle stage does not happen as there is a single vertex on the rim.  

The final stage of the walk is the descent down the volcano from level $k$ to level $b$.  There are many ways to descend to level $b$ but it is easy to count the number of possibilities.  Note that there are two reasons a non-backtracking walk would ascend up the volcano initially: either to get to level $b$ (if $a < b$ for example), or to stall, in the event that going down to level $b$ directly would result in a trip of length less than $c$.  
The corollary follows from consideration of the above analysis in each case, and writing the resulting formulas in terms of the character $\chi_K$.  
\end{proof}

As an example of how to compute these numbers ``by eye'', we now give the proof of Corollary \ref{gl2}.  This is the special case where $A = \C^2/\O_f^2$, and hence $\Aut(A) = \GL_2(\O_f)$. 

\begin{proof}[Proof of Corollary $\ref{gl2}$]
We need to show that if $g \mid \ell \mid f$, then 
\[r_K(f, g, f/\ell) =\begin{cases}
 1 & \mbox{if }g = \ell\\
  0 & \mbox{otherwise}.
  \end{cases}  \]
 By Lemma \ref{mult}, we may assume $f = p^a$, $g = p^b$, and $\ell = p^c$, with $a \geq c \geq b$.  Then if $b < c$, a vertex in level $a$ cannot walk to level $b$ in $a-c$ steps.  But if $b = c$, there is a unique (ascending) walk from a vertex of level $a$ to a vertex of level $b$ of length $a-b$, as claimed.     
\end{proof}

In the general case, the formula for $N(E,E')$ obtained by combining Theorem \ref{main} and Corollary \ref{explicit} does not seem to simplify much further.

\end{document}